\documentclass[oneside,notitlepage,12pt]{article}

\usepackage{amssymb}
\usepackage[leqno]{amsmath}
\usepackage{amsfonts}
\usepackage{amsopn}
\usepackage{amstext}
\usepackage{amsthm}

\usepackage{wasysym}

\usepackage{tikz}
\usetikzlibrary{cd}

\usepackage{enumitem}

\usepackage{verbatim}
\usepackage[colorlinks, backref]{hyperref}

\newcommand{\define}[2]{{\em #1}\index{#2}}

\usepackage{calrsfs}
\usepackage{clock} 

\hypersetup{
	pdftitle = {Uncountable homogeneous structures},
	pdfauthor = {Adam Bartoš, Wiesław Kubiś},
	colorlinks,
	linkcolor = [rgb]{0.8, 0, 0},
	citecolor = [rgb]{0, 0.6, 0.2},
}

\setlist{itemsep = 0pt}
\setlist[enumerate]{leftmargin=*} 
\setlist[enumerate, 1]{label=\upshape (\arabic*), ref=(\arabic*)}

\providecommand{\cal}{\mathcal}
\renewcommand{\Bbb}{\mathbb}

\newenvironment{pf}{\begin{proof}}{\end{proof}}

\newcommand{\Aaa}{{\cal{A}}}

\newcommand{\Ef}{{\cal{F}}}
\newcommand{\Gee}{{\cal{G}}}

\newcommand{\Kay}{{\cal{K}}}

\newcommand{\Zee}{{\Bbb{Z}}}

\newcommand{\Nat}{{\Bbb{N}}}
\newcommand{\Qyu}{{\Bbb{Q}}}

\newcommand{\Wu}{{\mathbb{W}}}

\newcommand{\al}{\alpha}

\newcommand{\sig}{\sigma}

\renewcommand{\phi}{\varphi}
\renewcommand{\rho}{\varrho}

\newcommand{\rest}{\restriction}

\newcommand{\ntr}{{n\in\omega}}

\newcommand{\loe}{\leq}
\newcommand{\goe}{\geq}

\newcommand{\subs}{\subseteq}
\newcommand{\sups}{\supseteq}

\newcommand{\ovr}{\overline}

\newcommand{\Cl}[1]{\overline{#1}}

\newcommand{\id}[1]{{\operatorname{id}_{#1}}} 

\newcommand{\rng}{\operatorname{rng}}

\newcommand{\by}[1]{/_{#1}}

\newcommand{\Fin}{\operatorname{Fin}}

\newtheorem{tw}{Theorem}[section]
\newtheorem{wn}[tw]{Corollary}
\newtheorem{lm}[tw]{Lemma}
\newtheorem{prop}[tw]{Proposition}

\theoremstyle{definition}
\newtheorem{df}[tw]{Definition}

\newtheorem{question}[tw]{Question}
\theoremstyle{remark}
\newtheorem{uwgi}[tw]{Remark}

\newcommand{\set}[1]{\{#1\}}
\newcommand{\setof}[2]{\{#1\colon #2\}}

\newcommand{\seq}[1]{\langle #1 \rangle}

\newcommand{\sett}[2]{\{#1\}_{#2}}
\newcommand{\sn}[1]{\{#1\}} 
\newcommand{\dn}[2]{\{#1,#2\}} 
\newcommand{\pair}[2]{\langle #1, #2 \rangle} 
\newcommand{\map}[3]{#1\colon #2 \to #3} 
\newcommand{\img}[2]{#1[#2]} 

\newcommand{\dpower}[2]{[#1]^{#2}}

\newcommand{\fin}[1]{[#1]^{<\omega}}

\newcommand{\fra}{Fra\"iss\'e}
\newcommand{\jon}{J\'onsson}

\newcommand{\U}{\mathbb U}
\newcommand{\M}{\mathbb M}

\providecommand{\nat}{\omega}

\newcommand{\aut}[1]{\operatorname{Aut}#1}

\newcommand{\iso}{\approx}

\newcommand{\emb}[1]{{\mathcal{E}{#1}}}

\newcommand{\cmp}{\circ} 

\newcommand{\separator}{\begin{center}
	*****
\end{center}}

\newcommand{\koment}[1]{}

\providecommand{\ar}{\arrow}

\newcommand{\cM}{\mathcal{M}}

\newcommand*{\card}[1]{\lvert#1\rvert} 
\newcommand*{\maps}{\colon} 

\title{Uncountable homogeneous structures}
\author{
{\sc Adam Barto\v{s}}\\
{\small Institute of Mathematics, Czech Academy of Sciences (CZECH REPUBLIC)} \\[-0.5ex]
{\small \tt bartos@math.cas.cz}
\and
{\sc Wies{\l}aw Kubi\'s}\\
{\small Institute of Mathematics, Czech Academy of Sciences (CZECH REPUBLIC)} \\[-0.5ex]
{\small \tt kubis@math.cas.cz}
}

\date{\clocktime\today}

\begin{document}

\maketitle

\begin{abstract}
	We study the existence of uncountable first-order structures that are homogeneous with respect to their finitely generated substructures. In many classical cases this is either well-known or follows from general facts, for example, if the language is finite and relational then ultrapowers provide arbitrarily large such structures. On the other hand, there are no general results saying that uncountable homogeneous structures with a given age exist. We examine the monoid of self-embeddings of a fixed countable homogeneous structure and, using abstract \fra\ theory, we present a method of constructing an uncountable homogeneous structure, based on the amalgamation property of this monoid.
	
	\ 
	

	
\end{abstract}

\newpage

\tableofcontents

\section{Introduction}

A first-order mathematical structure $\U$ is \emph{homogeneous}\footnote{Sometimes the adjective \emph{ultra-} is added to this definition in order to distinguish it from point homogeneity.} if every isomorphism between finitely generated substructures of $\U$ extends to an automorphism of $\U$. If such a $\U$ is additionally countably generated, it is uniquely determined up to isomorphism with respect to its age, by the classical \fra\ theory~\cite{Fraisse}, \cite{Hodges}. A natural extension of  \fra\ theory to uncountable structures was developed by \jon~\cite{Jon} and then complemented by Morley and Vaught~\cite{MorleyVaught}, obtaining uniquely determined structures with much stronger homogeneity. Namely, in these strongly homogeneous structures every isomorphism between substructures of smaller cardinality extends to an automorphism.
However, in most cases a cardinal arithmetic assumption of the form $2^{<\kappa} = \kappa$ is required for the existence of such structures.

In this note we present a new general method of constructing uncountable homogeneous structures, based on the properties of the monoid of self-embeddings of a fixed \fra\ limit.

Closely related to homogeneity is the extension property, sometimes called \emph{injectivity}, namely, given a finitely generated structure $B$ and its finitely generated substructure $A$, every embedding of $A$ extends to an embedding of $B$.
While a countable injective structure is homogeneous, this is not necessarily the case for uncountable ones. In fact, the other extreme may occur: An uncountable injective structure may be rigid, that is, its automorphism group may be trivial or minimal possible (for instance, it may happen that a non-trivial involution is always an isomorphism).
There are many well-known constructions of rigid injective structures, e.g., linear orders (rigid dense subsets of the real line, see~\cite{DroTru}, \cite{Scott}), graphs, Boolean algebras, etc. (general consistent examples are provided in~\cite{Ziemek}).

Given a homogeneous structure $\U$, let $\emb \U$ denote the monoid of all embeddings of $\U$ into itself, where the monoid operation is the composition of functions. It may happen that $\emb \U$ is trivial, that is, it only consists of the identity (equivalently, $\U$ is rigid -- see \ref{rmk:rigid}), or that $\emb \U$ is degenerate, i.e. it coincides with $\aut U$. A simple example of such a $\U$, in the relational language consisting of infinitely many unary predicates, is any structure in which each element belongs to exactly one of these predicates. Switching to an equivalent language of point colorings: each element has a unique color and each color appears at most once. The \fra\ limit of the class of all finite structures with this property is an infinite structure where all colors appear. Clearly, the only embedding of this structure into itself is the identity.
It turns out that such an example with a degenerate monoid of self-embeddings may be constructed even in the language of groups, see Example~\ref{SubSectgoijwegijw} below.

We address the complementary situation. Given a countable homogeneous structure (a \fra\ limit) $\U$ such that $\emb \U \ne \aut{\U}$, that is, at least one self-embedding $\map e \U \U$ is not an isomorphism, we consider building an uncountable structure as an increasing union of copies of $\U$.
It is well-known that \fra\ limits are stable under unions of countable chains. Therefore, using transfinite compositions of $e$, one can easily construct an uncountable structure with the same class (up to isomorphism) of finitely generated substructures as $\U$, satisfying the extension property (for details, see Proposition~\ref{thm:necessary} below).
Nevertheless, in such general constructions one cannot control automorphisms and one can even end up with a rigid structure with the extension property. 

Recall that the \emph{amalgamation property} says that two embeddings with the same domain can be completed by embeddings to a commutative square.
The \emph{age} (that is, the class of all finitely generated substructures) of a homogeneous structure clearly has the amalgamation property and in turn this property is crucial for the existence of the \fra\ limit---a homogeneous structure with the given age.

As it happens, the amalgamation property does not always extend to infinite structures, see Section~\ref{SECeroigjwroigj} below.
We say that a countable structure $M$ is an \emph{amalgamation base} if every two embeddings of $M$ into countable structures (from a given class) can be completed to a commutative square. In the case $M$ being a \fra\ limits, there is an easier characterization of this property. Namely, 
 if $M$ is a \fra\ limit, being an amalgamation base is equivalent to saying that for every two embeddings $\map f M M$, $\map g M M$ there exist embeddings $\map {f'} M M$, $\map {g'} M M$ such that $f' \cmp f = g' \cmp g$. Indeed, amalgamation gives embeddings $f_1$, $g_1$ into some countable structure $W$ and by the universality of $M$, there is an embedding $\map e W M$ so we may take $f' = e \cmp f_1$, $g' = e \cmp g_1$.
Summarizing, if $\U$ is a countable homogeneous structure then the information whether it is an amalgamation base is contained in the monoid of its self-embeddings.

Below is a simplified version of our main result.

\begin{tw}\label{th:mainsimple}
	Assume $\U$ is a countable homogeneous structure in a countable first-order language, such that $\U$ is an amalgamation base and $\emb \U \ne \aut \U$. Then there exists a homogeneous structure of size $\aleph_1$ that has the same age as $\U$.
\end{tw}

Note that in the case of a finite relational language, the ultraproduct construction gives us a required structure of size $2^{\aleph_0}$ (Corollary \ref{cor:uf}), and then we can take an elementary submodel to obtain a homogeneous structure of size $\aleph_1$. The point of Theorem \ref{th:mainsimple} is to deal with the countably infinite
languages, possibly with function symbols.

In the proof we are using abstract \fra\ theory (see e.g.~\cite{Kub40}), building countable sequences in the monoid $\emb \U$, obtaining new non-trivial self-embeddings of $\U$ that extend automorphisms.
The key property resembles \emph{naturality} (see~\cite{HoSh1, HoSh2} and Section~\ref{SECaReviwew} below), that is, uniform constructions of embeddings $\map e \U {F(\U)}$ such that every automorphism of $\U$ extends to $F(\U)$. Nonetheless, we do not require compositions to be preserved.

Let us mention that our method is fairly general and it does not use almost anything from model theory, except the fact that we can distinguish small structures from large ones and compute their cardinalities.

The paper is organized as follows. After Preliminaries, we review some known results. In Section~\ref{SECextensible} we present the main technical lemma proving our main result Theorem \ref{thm:main}. Section~\ref{SECeroigjwroigj} contains relevant examples. We finish with a discussion and open problems.

\section{Preliminaries}

Let $\delta$ be an ordinal. By a $\delta$-\emph{sequence} of structures we mean a function $\vec u$ with domain $\delta$, where $\vec u(\al)$ is a first-order structure denoted by $U_\al$ and for every $\al<\beta<\delta$ there is an embedding $\map{u_\al^\beta}{U_\al}{U_\beta}$. Furthermore, $u_\al^\al=\id{U_\al}$ and $u_\al^\gamma = u_\beta^\gamma \cmp u_\al^\beta$ whenever $\al<\beta<\gamma$. 

We mostly work with sequences of structures that are simply a chain, namely, $U_\al$ is a substructure of $U_\beta$ and $u_\al^\beta$ is the corresponding inclusion, whenever $\al<\beta$.

If $\delta>0$ is a limit ordinal, we can consider the \emph{colimit} of $\vec u$ which in the case of a chain is simply its union.
A sequence $\vec u$ with domain $\delta$ is \emph{continuous} if for every limit ordinal $\rho<\delta$, $U_\rho$ is the colimit of $\vec u \rest \rho$. If $\vec u$ is a chain, this simply means that $U_\rho = \bigcup \sett{U_\xi}{\xi < \rho}$ for every limit ordinal $\rho < \delta$.
For more details concerning functors and their limits/colimits we refer to~\cite{MacLane}. We are going to use only the very basic concepts of category theory.

Given a first-order structure $M$, we denote by $\Fin(M)$ its \emph{age}, namely, the class of all finitely generated structures embeddable into $M$. \fra\ Theorem states that a countably generated homogeneous structure is determined by its age. Furthermore, a class $\Ef$ of finitely generated structures is the age of a countable homogeneous structure if and only if it is hereditary, has the amalgamation property (see below), the joint embedding property, and has countably many isomorphism types. Such a class is called a \fra\  class and the countably generated homogeneous structure whose age is $\Ef$, is called the \fra\ limit of $\Ef$. Note that the last two conditions of being a \fra\ class are automatically fulfilled whenever the language is purely relational and finite, in which case the \fra\ limit is countable. The \fra\ limit of a \fra\ class $\Ef$ will typically be denoted by $\U$.

Recall that the \emph{amalgamation property} of a class $\Kay$ of structures means that for every two embeddings $\map fZX$, $\map gZY$ with $Z,X,Y \in \Kay$ there exist $W \in \Kay$ and embeddings $\map{f'}XW$, $\map {g'}YW$ such that $f' \cmp f = g' \cmp g$.

Given a class $\Ef$ of first-order structures, we denote by $\sig \Ef$ the class of all structures isomorphic to the unions of countable chains in $\Ef$. 
In the case that $\Ef$ is a \fra\ class, $\sig \Ef$ is just the class of all countably generated structures whose age is contained in $\Ef$.
We denote by $\Cl \Ef$ the class of all structures whose age is contained in $\Ef$.

In what follows, we shall always assume that the language of the structures under consideration is countable. In particular, countably generated structures are countable.

The monoid of self-embeddings of a structure $M$ will be denoted by $\emb{M}$.
A monoid will be called \emph{non-degenerate} if it is not a group, namely, it contains at least one non-invertible element. A homomorphism (typically, an embedding) will be called \emph{non-trivial} if it is not an isomorphism.

Recall that a structure $X$ is \emph{injective} if for every two structures $A \subs B$ in $\Fin(X)$, every embedding $\map e A X$ can be extended to an embedding $\map {\tilde{e}} B X$.
The limit $\U$ of a Fraïssé class $\Ef$ can also be characterized as being injective, countable and with age $\Ef$.
Because of this and the fact that injectivity is preserved by colimits of sequences, we obtain the following.

\begin{lm} \label{thm:limitU}
	Let $\delta$ be any ordinal of cofinality $\omega$.
	The colimit of any $\delta$-sequence $\vec u$ of structures, each of them being isomorphic to $\U$, is again isomorphic to $\U$.
\end{lm}
\begin{pf}
	Without loss of generality, $\delta = \omega$ as any cofinal subsequence of $\vec u$ has the same colimit as $\vec u$.
	Then clearly the colimit $U_\infty$ is countable.
	As $\map {u_0^\infty} \U U_\infty$, we have $\Fin(U_\infty) \sups \Ef$.
	Since for every embedding of a finitely generated structure $\map f A U_\infty$ there is an embedding $\map {f'} A U_n$ for some $n \in \omega$, we have $\Fin(U_\infty) \subs \Ef$.
	Moreover, for every $B \sups A$ in $\Ef$ there is an embedding $\map {f''} B U_n$ extending $f'$ since $U_n \iso \U$ is injective.
	Hence, $\tilde{f} = u^\infty_n \cmp f''$ extends $f$, and therefore $U_\infty$ is injective.
\end{pf}

\begin{prop} \label{thm:necessary}
	The following conditions are equivalent.
	\begin{enumerate}
		\item The monoid $\emb \U$ is non-degenerate, i.e., there is a non-trivial self-embedding $\map e \U \U$.
		\item There is an uncountable injective structure $X$ with age $\Ef$.
	\end{enumerate}
\end{prop}
\begin{pf}
	If $\map e \U \U$ is a non-trivial embedding, with the use of Lemma~\ref{thm:limitU} we can build a continuous $\omega_1$-sequence $\vec{u}$ such that $U_\al = \U$ and $u^{\al + 1}_\al = e$ for every $\al \in \omega_1$. Then consider its colimit $U_{\omega_1}$.
	Clearly $U_{\omega_1}$ is of cardinality $\aleph_1$.
	As in the proof of Lemma~\ref{thm:limitU} we argue that $U_{\omega_1}$ is injective and has age $\Ef$.
	
	Conversely, suppose $X$ is an uncountable injective structure with age $\Ef$.
	We can write $\U$ as an increasing union of finitely generated substructures $\bigcup_{n \in \omega} A_n$.
	There is an embedding $\map {f_0} {A_0} \U$.
	Since $X$ is injective, we can find a sequence of embeddings $\map {f_n} {A_n} X$, $n \in \omega$, each extending the previous one, and so $f = \bigcup_{n \in \omega} f_n$ embeds $\U$ into $X$.
	As $X$ is uncountable, there is a countable substructure $Y$ such that $\img f \U \subsetneq Y \subs X$.
	Since $\U$ is universal, $Y$ can be embedded into $\U$, and we obtain a non-trivial embedding $\map e \U \U$.
\end{pf}

\begin{uwgi} \label{rmk:rigid}
	Note that if $\U$ is rigid, i.e. $\aut \U$ is trivial, then $\emb \U$ is trivial as well.
	If there was a non-trivial embedding $\map e \U \U$, then there would be an automorphism of $\U$ mapping a substructure $A$ generated by a point $x \in \U \setminus \rng(e)$ so its isomorphic copy in $e[\U]$.
	Hence, $\emb \U = \aut \U = \set{\id{\U}}$.
\end{uwgi}

\section{A review of known results}\label{SECaReviwew}

As we have already mentioned, there are several well-known methods of constructing uncountable homogeneous structures, sometimes of arbitrary large size, sometimes possessing much stronger homogeneity, namely, with respect to infinitely generated substructures.

\subsection{Ultraproducts}\label{SSecERTrt}

As one may expect, ultraproducts often preserve homogeneity:

\begin{prop}\label{PROPsdohsdoiag} Let $\kappa\ge \aleph_0$ be a cardinal and suppose that 
 $\langle M_\al:\,\al<\kappa\rangle$ is a sequence of homogeneous structures in a fixed finite relational language. Then for every ultrafilter $p$ on $\kappa$, the ultraproduct $(\prod_{\al<\kappa}M_\al) \by p$ is homogeneous.
\end{prop}

\begin{pf}
	Let $\M = (\prod_{\al<\kappa}M_\al) \by p$. Given $x \in \prod_{\al<\kappa}M_\al$, we denote by $x \by p$ its equivalence class in $\M$. Given $S \subs \prod_{\al<\kappa}M_\al$, we write $S \by p = \setof{x \by p}{x \in S}$.
	Let $\map g {\tilde A}{\tilde B}$ be a finite isomorphism inside $\M$. We select sets of representatives $A$, $B$, so that $\tilde A = A \by p$, $\tilde B = B \by p$ and $(A \cap B) \by p = \tilde A \cap \tilde B$.
	Since the language is finite and relational,
	without loss of generality (discarding a set of coordinates outside of $p$) we may assume that $A(\al) := \setof{a(\al)}{a \in A}$ is canonically isomorphic to $A$. The same for $B$.
	In particular, for each $\al<\kappa$ the map $\map{g_\al}{A(\al)}{B(\al)}$, defined by $g_\al(a(\al)) = g(a)(\al)$ is a finite isomorphism.
	By the homogeneity of $M_\al$, there exists $h_\al \in \author{M_\al}$ extending $g_\al$. Finally, the product $h = \prod_{\al<\kappa}h_\al$ induces an automorphism of $\M$ extending $g$.
\end{pf}

\begin{wn}\label{cor:uf}
	Let $\Ef$ be a relational \fra\ class in a finite language and let $\U$ be its \fra\ limit.
	If $\U$ is infinite then for every non-principal ultrafilter $p$ on $\omega$, the ultrapower $\U^\omega \by p$ is homogeneous, of cardinality continuum, and has age $\Ef$.
\end{wn}

Proposition~\ref{PROPsdohsdoiag} actually shows that there are arbitrarily large homogeneous structures with age $\Ef$ and of any given infinite cardinality (by taking 
a large enough ultrapower and a standard closing-off argument, in order to make it smaller---see Section~\ref{SubSecteohwroghweoig} for more details).
Note that if the language is infinite or just not relational, ultrapowers of a \fra\ limit may have bigger ages than  $\U$. For example, any non-trivial ultrapower of Hall's group (the \fra\ limit of finite groups) contains elements of infinite order, therefore is not locally finite.
On the other hand, Keisler~\cite{Keisler64} proved that arbitrary ultraproducts are $\aleph_0$-saturated. Note that a countably infinite $\aleph_0$-saturated structure is the \fra\ limit of its age.

\begin{uwgi}
Our original approach to construct an uncountable homogeneous structure of the prescribed age was to use \emph{simplified $(\omega, 1)$-morasses} of Velleman~\cite{Velleman84}, which, unlike ultrapower, give control over the age, and can yield an injective structure, but we were not able to prove homogeneity.

In more detail, a simplified $(\omega, 1)$-morass can be specified by the following data. First, there is an increasing sequence $\set{n_i}_{i \in \omega}$ of natural numbers with two sequences of one-to-one maps $\vec{u} = \set{u_i\maps n_i \to n_{i + 1}}$ and $\vec{v} = \set{v_i\maps n_i \to n_{i + 1}}$ that `split' every $n_i$, meaning that $n_i$ can be written as the disjoint union of subsets $A < B$, with $B \neq \emptyset$, and $n_{i + 1}$ can be written as the disjoint union of $A' < B'_0 < B'_1$, and $u_i$ is the increasing bijection $n_i = A \cup B \to A' \cup B'_0 \subs n_{i + 1}$ (which is the identity), while $v_i$ is the increasing bijection $n_i = A \cup B \to A' \cup B'_1 \subs n_{i + 1}$.

Second, there are families $\Gee_i$ of one-to-one maps $n_i \to \omega_1$, $i \in \omega$, satisfying certain conditions that will allow us to turn $\omega_1$ into a structure once we turn every $n_i$ into a structure and $\vec{u}, \vec{v}$ into sequences of embeddings.
Namely, for every $g \in \Gee_i$ we copy the structure from $n_i$ to $g[n_i]$ via $g$, and this will correctly define a structure $\Wu$ on $\omega_1$.
Note that given a structure on $n_i$, to specify a structure on $n_{i + 1}$ making $u_i, v_i$ embeddings means to provide a strong amalgamation of the given structures $A' \cup B'_0$ and $A' \cup B'_1$ over $A'$.
If we do it in a way that $\vec{u}$ becomes a Fraïssé sequence (cf.~\cite[Section 3]{Kub40}) in $\Ef$, then the structure $\Wu$ on $\omega_1$ will be injective with $\Fin(\Wu) = \Fin(\U)$.
\end{uwgi}

\subsection{Natural constructions}

Let $\Kay$ be a class of structures in a fixed first-order language. According to Hodges and Shelah (\cite{HoSh1, HoSh2}) a \emph{natural construction} is a map $F$ assigning to each $A \in \Kay$ a bigger structure $F(A) \in \Kay$, possibly in an expanded language, such that every automorphism $h$ of $A$ extends to an automorphism $F(h)$ of $F(A)$ and the composition is preserved.
The papers~\cite{HoSh1, HoSh2} studied relations between naturality and definability.
For example, free constructions (left adjoints to a forgetful functor) are natural (even much more, $F$ is a functor on all homomorphisms), while algebraic closures of fields are not.

Suppose that $\map F \Kay \Kay$ is a natural construction, with no expansion of the language. 
Then we have a canonical inclusion $A \subs F(A)$ up to isomorphism for every $A \in \Kay$. To simplify this explanation, let us assume that $A\subseteq 
F(A)$ literally. Hence for any fixed $A \in \Kay$, we can iterate $F$ to obtain a chain
$$A \subs F(A) \subs F^2(A) \subs \cdots.$$
Let $F^\omega(A) := \bigcup_{\ntr} F^n(A)$.
Naturality implies that for each $n$, every automorphism of $F^n(A)$ extends to an automorphism of $F^\omega(A)$.

Now suppose that $\Kay$ is stable under unions of $\omega$-chains. Then $F^\omega(A) \in \Kay$ and we can iterate $F$ in a transfinite way, obtaining an $\omega_1$-chain $\sett{F^\al(A)}{\al<\omega_1}$, taking unions at limit steps. Now, if $\Kay$ consists of countable structures and each canonical inclusion $X \subs F(X)$ is non-trivial (not the identity) then the structure $F^{\omega_1}(A) := \bigcup_{\al<\omega_1}F^\al(A)$ is uncountable.
Furthermore, again using naturality, we see that for each $\al<\omega_1$, every automorphism of $F^{\al}(A)$ extends to an automorphism of $F^{\omega_1}(A)$. Since every isomorphism between finitely generated substructures of $F^{\omega_1}(A)$ is ``captured'' by some $F^\al(A)$, we could conclude that $F^{\omega_1}(A)$ is homogeneous, as long as $F^\al(A)$ is homogeneous for an unbounded set of $\al<\omega_1$.

Summarizing, if $A$ is a \fra\ limit and $F(A)$ is isomorphic to $A$ then the same holds for every countable $\al$, and therefore we have a method of constructing uncountable homogeneous structures. Concrete cases are described below.

In fact, full naturality is not needed, in the construction above we only need to know that automorphisms extend. This will be further explored in Section~\ref{SECextensible}.

\subsubsection{Rigid moieties}

Let $\Ef$ be a \fra\ class and let $\U$ be its \fra\ limit. A \emph{rigid moiety} is a pair $\pair A B$ where $A$ is a substructure of $B$, both $A$ and $B \setminus A$ are countable infinite with age contained in $\Ef$, and every automorphism of $A$ extends uniquely to an automorphism of $B$. Obviously, this becomes interesting when $B \iso \U$ and even more when $A \iso \U$.
This concept was formally introduced by Bilge and Jaligot~\cite{BilJal}, although the first existence result was due to Henson~\cite{Henson}, concerning graphs.
Jaligot~\cite{Jaligot1} showed that rigid moieties exist in the class of tournaments, and later
Bilge and Jaligot~\cite{BilJal} showed the same for the class of $K_n$-free graphs for every $n>2$. In particular, there exists a rigid moiety $\pair AB$, where both $A$, $B$ are isomorphic to Henson's universal homogeneous $K_n$-free graph.

Given a rigid moiety $\pair AB$ with both $A$ and $B$ isomorphic to the \fra\ limit $\U$, one immediately gets a natural construction on the class $\Kay$ consisting of all structures isomorphic to $\U$.

Apart from the case of Henson's universal homogeneous $K_n$-free graph described above, it is not known which \fra\ classes admit rigid moieties with both components isomorphic to the \fra\ limit.

\subsubsection{Kat\v etov functors}

We briefly recall the main concept of~\cite{KubMas} (going back to Kat\v etov~\cite{Kat}) and its use for proving the existence of arbitrarily large homogeneous structures.

Fix a \fra\ class $\Ef$. A \emph{Kat\v etov functor} on $\Ef$ is a functor $\map K \Ef {\sig \Ef}$, together with a collection of embeddings $\sett{\map {\eta_X}X{K(X)}}{X \in \Ef}$, satisfying the following conditions.
\begin{enumerate}[itemsep=0pt]
	\item \label{Kzero} Given an embedding $\map e XY$ with $X, Y \in \Ef$, the diagram
	$$\begin{tikzcd}
		X \ar[d, hook, "e"'] \ar[r, hook, "\eta_X"] & K(X) \ar[d, hook, "K(e)"] \\
		Y \ar[r, hook, "\eta_Y"'] & K(Y)
	\end{tikzcd}$$
	commutes.
	
	\item Given $X \in \Ef$ and a one-point extension $Y \sups X$, there exists an embedding $\map f Y{K(X)}$ such that $f \rest X = \eta_X$.
\end{enumerate}
In other words, the embedding $\eta_X$ of $X$ into $K(X)$ realizes all one-point extensions of $X \in \Ef$. We have to make precise the meaning of $K$ being a functor. Namely, this is a covariant functor on $\Ef$ with all embeddings into $\sig \Ef$ with all embeddings. A more general variant was considered in~\cite{KubMas}, where $\Ef$ and $\sig \Ef$ were treated as categories with all homomorphisms, however then $K$ was supposed to preserve embeddings.
The collection $\sett{\eta_X}{X \in \Ef}$ is nothing but a natural transformation from the inclusion $\Ef \subs \sig \Ef$ to the functor $K$. In the case that $\eta_X$ is the inclusion of $X$ into $K(X)$, we obviously have a natural construction in the sense of Hodges and Shelah, mentioned above.

It has been proved in~\cite{KubMas} that every Kat\v etov functor canonically extends to $\sig \Ef$, so that condition \ref{Kzero} holds for arbitrary embeddings between structures from $\sig \Ef$.
Then for any $X \in \sig\Ef$, we have a sequence of embeddings
$$\begin{tikzcd}
	X \ar[r, hook, "\eta_X"] & K(X) \ar[r, hook, "\eta_{K(X)}"] & K^2(X) \ar[r, hook] & \cdots
\end{tikzcd}$$
whose colimit $K^\omega(X)$ is isomorphic to the \fra\ limit of $\Ef$.
Let $\map{\eta^\omega_X}{X}{K^\omega(X)}$ be the co-limiting embedding, that is, part of the co-limiting co-cone over the above sequence.
Then $K^\omega$ together with $\eta^\omega$ satisfies condition (1) above and a variant of condition (2), making it a Kat\v etov functor whose values on all objects are isomorphic to the \fra\ limit $\U$ of $\Ef$.

Hence if $\map{\eta^\omega_\U}{\U}{K^\omega(\U)}$ is non-trivial, then we have a way of constructing uncountable homogeneous structures with age $\Ef$.
Actually, a Kat\v etov functor extends uniquely to the class $\ovr \Ef$ of all structures with age in $\Ef$. It may happen that $\ovr \Ef = \sig \Ef$ (see the example of locally cyclic torsion-free groups in Section~\ref{SECeroigjwroigj}); otherwise there is a chance of building arbitrarily large structures, as long as the corresponding natural transformation is a non-trivial embedding, so that the transfinite chain still grows.

\section{Extensibility}\label{SECextensible}

We now extract the key property of self-embeddings that allows for constructions of uncountable homogeneous structures. Later we introduce a parametrized version, relevant for our main result.

An embedding $\map e XY$ is said to be \emph{extensible} if for every automorphism $h$ of $X$ there exists an automorphism $\tilde{h}$ of $Y$ such that $\tilde{h} \cmp e = e \cmp h$,
	$$\begin{tikzcd}
		X \ar[d, "h"'] \ar[r, hook, "e"] & Y \ar[d, "\tilde{h}"] \\
		X \ar[r, hook, "e"'] & Y
	\end{tikzcd}$$
as shown in the diagram.

Of course, every identity is extensible (trivially) and the composition of two extensible embeddings is extensible. Furthermore, if $\U$ is a homogeneous structure and $F \subs \U$ is finitely generated, then the inclusion $F \subs \U$ is extensible.
Clearly, every natural construction in the sense of Hodges and Shelah~\cite{HoSh1, HoSh2} gives extensible embeddings. In fact,  in~\cite{HoSh2} one can find a weakening, called \emph{weak naturality}, asserting that automorphisms extend to automorphisms, without requiring that the composition be preserved. This is exactly what we need, and we define it locally, for a pair of concrete structures, see Definition \ref{def:parext}.

A non-trivial (i.e. not surjective) extensible self-embedding of a \fra\ limit leads to an uncountable homogeneous structure, namely:

\begin{prop}\label{prop:limit}
	Suppose that $\U$ is a \fra\ limit and $\map e \U\U$ is a non-trivial extensible embedding. Furthermore, suppose that we are given a
	sequence $\vec u=\seq{U_\alpha: \alpha<\omega_1}$ such that:
	\begin{enumerate}
\item Each $U_\alpha$ is isomorphic to $\U$, through an isomorphism $u_\alpha\maps \U \to U_\alpha$
\item For any $\alpha$, there is an embedding $u_{\al}^{\al+1}\maps U_\alpha \to U_{\alpha+1}$ such that the following diagram commutes:

$$\begin{tikzcd}
		\U \ar[d, "u_\alpha"'] \ar[r, hook, "e"] & \U \ar[d, "u_{\alpha+1}"] \\
		U_\alpha \ar[r, hook, "u_{\al}^{\al+1}"'] & U_{\alpha+1}
	\end{tikzcd}$$
\item For any limit $\delta\in (0,\omega_1)$, we have that $U_\delta$ is the colimit of the system of embeddings
$\seq{U_\alpha, u_{\al}^{\al+1}: \alpha<\delta}$. In particular, $U_\delta$ is unique up to isomorphism and isomorphic to $\U$. (The latter follows
from \fra\ Theorem.)
\end{enumerate}
Then the colimit of $\vec u$ is an uncountable homogeneous structure.
\end{prop}
 
The fact that the colimit $U^\ast$ of $\vec u$ exists and is unique up to isomorphism follows from the general categorical considerations by composing embeddings. It also has the same age as $\U$. 
The fact that it is homogeneous follows since every $U_\alpha$ is homogeneous, $e$ is extensible and the age of $U^\ast$ is the same as that of each $U_\alpha$. The fact that
$U^\ast$ is uncountable follows from the fact that $e$ is not surjective. We shall explain these details in the proof of Proposition~\ref{thm:countably_extensible}, which is
an extension of Proposition~\ref{prop:limit}. That lemma will also imply that a sequence $\vec u$ as above can be constructed by transfinite recursion.

Motivated by the above, we pose the following question:

\begin{question}
	Let $\U$ be a countable homogeneous structure (the \fra\ limit of its age). What conditions on the age of $\U$ guarantee that there exists a non-trivial extensible embedding $\map e \U \U$?
\end{question}

We now extend the definition of extensible embeddings to a parametrized version.

\begin{df}\label{def:parext}
Given a subset (not necessarily a subgroup) $G\subseteq \aut X$, we shall say that $\map e X Y$ is \emph{$G$-extensible} if for every $h \in G$ there is $\tilde{h} \in \aut Y$ such that $e \cmp h = \tilde{h} \cmp e$. 
\end{df}

So $e$ is extensible if it is $\aut X$-extensible. When applied to  \fra\ limits, 
the parametrized variant of extensibility can be used for constructing uncountable homogeneous structures as unions of extensible chains.

\begin{df}
	By an \emph{extensible chain} of length $\delta$ we mean a sequence $\set{U_\alpha, G_\alpha, s_\alpha^\beta}_{\alpha<\beta<\delta}$ such that
	\begin{enumerate}
		\item $\set{U_\alpha}_{\alpha < \delta}$ is a strictly increasing continuous chain of homogeneous infinite structures in $\Cl{\Ef}$, i.e. $U_\alpha \subsetneq U_{\alpha + 1}$ for every $\alpha < \delta$, and $U_\beta = \bigcup_{\alpha < \beta} U_\alpha$ for every limit $\beta < \delta$;
		\item every $G_\alpha$ is a family of automorphisms of $U_\alpha$ of cardinality $\leq \card{U_\alpha}$ witnessing its homogeneity, i.e. every isomorphism of finitely generated substructures of $U_\alpha$ can be extended to an automorphism from $G_\alpha$;
		\item for every $\alpha < \beta < \delta$ the map $s_\alpha^\beta\maps G_\alpha \to G_\beta$ witnesses that the inclusion $U_\alpha \subs U_\beta$ is $G_\alpha$-extensible, i.e. $h \subs s_\alpha^\beta(h)$ for every $h \in G_\alpha$;
		\item the family $\set{s_\alpha^\beta}_{\alpha < \beta < \delta}$ is coherent, i.e. $s^\gamma_\beta \cmp s^\beta_\alpha = s^\gamma_\alpha$ whenever $\alpha < \beta < \gamma < \delta$.
	\end{enumerate}
	An \emph{extensible $\U$-chain} is an extensible chain $\set{U_\alpha, G_\alpha, s_\alpha^\beta}_{\alpha<\beta<\delta}$ such that $U_\alpha \iso \U$ for every $\alpha < \delta$.
\end{df}

Recall that $\Fin(M)$ denotes the age of a first-order structure $M$.

\begin{lm}\label{thm:limit_extension}
	Every extensible chain $\set{U_\alpha, G_\alpha, s_\alpha^\beta}_{\alpha<\beta<\delta}$ of a limit length $\delta$ can be extended to an extensible chain of length $\delta + 1$.
	Moreover, the homogeneous limit structure $U_\delta$ satisfies $\card{U_\delta} = \card{\delta} + \sup_{\alpha<\delta} \card{U_\alpha}$ and $\Fin(U_\delta) = \bigcup_{\alpha < \delta} \Fin(U_\alpha)$.
\end{lm}
\begin{pf}
	We put $U_\delta = \bigcup_{\alpha<\delta} U_\alpha$, and for every $\alpha < \delta$ and $h \in G_\alpha$ we put $s_\alpha^\delta(h) = \bigcup_{\beta\in(\alpha, \delta)} s_\alpha^\beta(h)$.
	By the coherence condition, $\set{s_\alpha^\beta}_{\beta\in(\alpha, \delta)}$ is an increasing chain, and $s_\alpha^\delta(h)$ is well-defined, as depicted in the following diagram.
	\[\begin{tikzcd}
		U_\al \ar[d, "h"'] \ar[r, "\subsetneq"] & U_{\al+1} \ar[d, "s_\al^{\al+1}(h)"] \ar[r, "\subsetneq"] & \cdots \ar[r, "\subsetneq"] & U_\delta \ar[d, dashed, "s_\al^\delta(h)"] \\
		U_\al \ar[r, "\subsetneq"'] & U_{\al+1} \ar[r, "\subsetneq"'] & \cdots \ar[r, "\subsetneq"'] & U_\delta 
	\end{tikzcd}\]
	It is also easy to see that $s_\alpha^\delta(h) \in \aut{U_\delta}$.
	Finally, we put $G_\delta = \bigcup_{\alpha < \delta} s^\delta_\alpha[G_\alpha]$.
	
	The limit structure is homogeneous and we have $\Fin(U_\delta) = \bigcup_{\alpha < \delta} \Fin(U_\alpha)$ since every finitely generated substructure of $U_\delta$ is already a substructure of $U_\alpha$ for some $\alpha < \delta$.
	Hence, an isomorphism of finitely generated substructures of $U_\delta$ is already inside of some $U_\alpha$, and so extends to an automorphism $h \in G_\alpha$, which then extends to $s^\delta_\alpha(h) \in G_\delta$.
	
	Obviously we have $\card{U_\delta} \geq \sup_{\alpha < \delta} \card{U_\alpha}$, and we have $\card{U_\delta} \geq \card{\delta}$ since every inclusion $U_\alpha \subs U_{\alpha+1}$ is proper.
	On the other hand, $\card{\bigcup_{\alpha<\delta} U_\alpha} \leq \card{\delta} \cdot \sup_{\alpha<\delta} \card{U_\alpha}$.
	Hence, $\card{U_\delta} = \card{\delta} + \sup_{\alpha<\delta} \card{U_\alpha}$.
	Similarly, we obtain $\card{G_\delta} \leq \card{\delta} + \sup_{\alpha<\delta} \card{G_\alpha} \leq \card{\delta} + \sup_{\alpha<\delta} \card{U_\alpha} = \card{U_\delta}$.
	
	Note that the family $\set{s_\alpha^\beta}_{\alpha<\beta<\delta+1}$ is still coherent since 
	\[\textstyle
		s^\delta_\beta \cmp s^\beta_\alpha = (\bigcup_{\gamma \in (\beta, \delta)} s^\gamma_\beta) \cmp s^\beta_\alpha = \bigcup_{\gamma \in (\beta, \delta)} (s^\gamma_\beta \cmp s^\beta_\alpha) = \bigcup_{\gamma \in (\beta, \delta)} s^\gamma_\alpha = s^\delta_\alpha.
	\]
	The remaining properties are easy to see.
\end{pf}

\begin{prop} \label{thm:countably_extensible}
	Suppose that $\U$ is a \fra\ limit of $\Ef$ such that for every countable $G \subseteq \aut \U$ there exists a non-trivial $G$-extensible embedding $\map e \U \U$.
	Then there exists an extensible $\U$-chain of length $\omega_1$, and consequently there exists an uncountable homogeneous structure with age $\Ef$.
\end{prop}
\begin{pf}
	First observe that by the assumption, for every $U_\alpha \iso \U$ and countable $G_\alpha \subs \aut{U_\alpha}$ there exists $U_{\alpha + 1} \supsetneq U_\alpha$ isomorphic to $U_\alpha$ such that the inclusion $U_\alpha \subsetneq U_{\alpha + 1}$ is $G_\alpha$-extensible.
	This is done by fixing an isomorphism $i\maps U_\alpha \to \U$, putting $G = i G_\alpha i^{-1}$, and finding a suitable isomorphism $j$, as in the following diagram.
	\[\begin{tikzcd}
		\U \ar[r, "e"] & \U \ar[d, "j"] \\
		U_\alpha \ar[u, "i"] \ar[r, "\subsetneq"'] & U_{\alpha + 1}
	\end{tikzcd}\]
	
	To build an extensible $\U$-chain of length $\omega_1$, we start with $U_0 = \U$ and $G_0 \subs \aut{\U}$ being a countable family witnessing the homogeneity of $\U$.
	Suppose we already have $\set{U_\alpha, G_\alpha, s_\alpha^\beta}_{\alpha < \beta \leq \delta}$ for some $\delta < \omega_1$.
	We use the above observation to obtain a $G_\delta$-extensible inclusion $U_\delta \subsetneq U_{\delta + 1}$.
	Then we let $s_\delta^{\delta + 1}$ be any map witnessing the $G_\delta$-extensibility and we define $G_{\delta + 1}$ as the union of $s_{\delta}^{\delta+1}[G_\delta]$ and any countable family of automorphisms of $U_{\delta + 1}$ witnessing its homogeneity.
	Finally, we put $s^{\delta+1}_\alpha = s^{\delta+1}_\delta \cmp s^\delta_\alpha$ for every $\alpha < \delta$.
	This inductively defines the chain at successor steps.
	Lemma~\ref{thm:limit_extension} takes care of the limit steps.
	Here we note that for limit $\delta < \omega_1$, $U_\delta$ remains countable homogeneous with age $\Ef$, and so is isomorphic to $\U$ by the uniqueness of the Fraïssé limit, and the construction may continue.
	This finishes the construction of the extensible $\U$-chain $\set{U_\alpha, G_\alpha, s_\alpha^\beta}_{\alpha<\beta<\omega_1}$.
	
	Finally, again by Lemma~\ref{thm:limit_extension}, the chain can be extended to a homogeneous structure $U_{\omega_1} = \bigcup_{\alpha<\omega_1} U_\alpha$ of cardinality $\aleph_1$ and age $\Ef$.
\end{pf}

\begin{tw} \label{thm:main} Suppose that $\U$ is a \fra\ limit and $\cM$ is a non-degenerate sub-monoid of $\emb\U$ containing all automorphisms of $\U$. If $\cM$ has the amalgamation property then there exists a homogeneous structure $\U_{\omega_1}$ of cardinality $\aleph_1$ that is the union of a chain of structures isomorphic to $\U$.
\end{tw}

\begin{pf}
	Fix a countable group $G \loe \aut \U$. Using a simple closing-off argument, we can find a countable non-degenerate sub-monoid $\cM_0$ containing $G$ and having the amalgamation property.
	So $\cM_0$ admits a \fra\ sequence $\vec u$, as below.
	$$\begin{tikzcd}
		\U \ar[r, "u_0^1"] & \U \ar[r, "u_1^2"] & \U \ar[r] & \cdots
	\end{tikzcd}$$
	This means (cf.~\cite[Section 3]{Kub40}) that for every $\ntr$, for every embedding $\map f {\U}{\U}$ there exist $m>n$ and an embedding $\map g{\U}{\U}$ such that $g \cmp f = u_n^m$.
	In particular, taking $f$ to be any nontrivial embedding, we see that there is $n>0$ such that $u_0^n$ is nontrivial.
	
	The colimit of $\vec u$ is again $\U$, by Lemma~\ref{thm:limitU}.
	Denote by $u_0^\infty$ the colimiting\footnote{Once the sequence above is a chain, $u_0^\infty$ is simply a suitable inclusion of a copy of $\U$ into another copy of $\U$; otherwise $u_0^\infty$ is defined up to isomorphism as part of the universal co-cone.} self-embedding of $\U$, which is necessarily nontrivial, because it is of the form $u_n^\infty \cmp u_0^n$ for $n$ big enough so that $u_0^n$ is nontrivial.
	Homogeneity of a \fra\ sequence (cf.~\cite[Proposition 3.14]{Kub40})
	says, in particular, that for every $g \in G$ there exists an automorphism $\tilde g$ of $\U$ such that the diagram
	$$\begin{tikzcd}
		\U \ar[d, "g"'] \ar[rrr, "u_0^\infty"]&&& \U \ar[d, "\tilde{g}"] \\
		\U \ar[rrr, "u_0^\infty"']&&& \U
	\end{tikzcd}$$
	commutes.
	Thus, the colimiting embedding $\map{u_0^\infty}{\U}{\U}$ is $G$-extensible.
	It follows that the assumptions of Proposition~\ref{thm:countably_extensible} are fulfilled.
\end{pf}

\begin{uwgi} \label{rm:amalgamation}
	Note that if the whole class of countable structures $\sig\Ef$ has the amalgamation property, then so does the monoid $\emb\U$:
	for every $f, g \in \emb\U$ there is $X \in \sig\Ef$ and embeddings $\map {f', g'} \U X$ such that $f' \cmp f = g' \cmp g$.
	Since $\U$ is universal in $\sig\Ef$, there is an embedding $\map e X \U$, and so we have $(e \cmp f') \cmp f = (e \cmp g') \cmp g$ in the monoid.
	In fact, $\emb\U$ has AP if and only if $\U$ is an amalgamation base in $\sig\Ef$: by the universality of $\U$, every pair of embeddings $\map f \U X$, $\map g \U Y$ in $\sig\Ef$ can be extended to a pair $f', g' \in \emb\U$.	
\end{uwgi}

\subsection{Towards the converse}\label{SubSecteohwroghweoig}

Let $\Ef$ and $\U$ be as above, i.e. $\Ef$ is a Fraïssé class with a countable limit $\U$. Assume $\Wu$ is an uncountable homogeneous structure with age $\Ef$.
The following considerations show that in this case there exist a non-trivial ``almost'' extensible embedding.

Let us fix a large enough cardinal $\theta$ and a countable elementary submodel $M$ of $\pair{H(\theta)}{\in}$ such that $\Wu \in M$. Then $\Ef, \U \in M$.
Let $W_0 := \Wu \cap M$. Then $W_0 \iso \U$, because $M$ ``knows'' all isomorphisms between finitely generated substructures of $W_0$ and by the homogeneity of $\Wu$, $M$ ``knows'' that each of them can be extended to an automorphism of $\Wu$.
Note that $M \cap \aut \Wu$ is a countable group and
$$G_M := \setof{g \rest W_0}{g \in M \cap \aut \Wu}$$
is a countable subgroup of $\aut W_0$.
Now choose a countable elementary submodel $M'$ of $H(\theta)$ such that $M \in M'$. In particular, $M \cap \aut \Wu \in M'$. Let $W_1 := \Wu \cap M'$. Then again $W_1 \iso \U$ and each $g \in G_M$ extends to an automorphism of $W_1$, because $W_1$ is invariant with respect to all automorphisms from $M' \cap \aut \Wu$. It follows that the inclusion $W_0 \subs W_1$ is $G_M$-extensible.

Using transfinite induction we obtain a continuous $\omega_1$-chain $\sett{W_\al}{\al<\omega_1}$ of structures isomorphic to $\U$ and a chain $\sett{G_\al}{\al<\omega_1}$ of countable subgroups of $\aut \Wu$ such that for each $\al<\beta$ the inclusion $W_\al \subs W_\beta$ is $(G_\al \rest W_\al)$-extensible, where $G_\al \rest W_\al := \setof{g \rest W_\al}{g \in G_\al} \loe \aut W_\al$. 

\begin{uwgi}
	We do not know whether, in the setting above, for every countable group $G \loe \aut \U$ there is a non-trivial $G$-extensible embedding of $\U$, but we know it for every countable $G \loe \aut \Wu$, where $G$-extensibility is understood as $(G \rest W)$-extensibility for a suitable $W \loe \Wu$.
\end{uwgi}

Using the notion of extensible chains, the closing-off construction can be made precise, giving a characterization existence of an uncountable homogeneous structure with the age of a Fraïssé limit.

\begin{lm} \label{thm:closing_off}
	Let $\Wu$ be a homogeneous structure with age $\Ef$.
	For every infinite $X \subs \Wu$ and $H_0 \subs \aut{\Wu}$ such that $\card{H_0} \leq \card{X}$ there exist a homogeneous structure $U$ with $X \subs U \leq \Wu$ and $\card{X} = \card{U}$, and $H \subs \aut{\Wu}$ such that $H_0 \subs H$ and $\card{H} \leq \card{X}$. Furthermore, $U$ is invariant under the action of $H$ and this action witnesses that $U$ is homogeneous.
\end{lm}
\begin{pf}
	Let $X' = \bigcup\set{h^k[X]: h \in H_0,\, k \in \Zee}$, i.e. the closure of $X$ under the action by the automorphisms from $H_0$, and let $X_0 \leq \Wu$ be the substructure generated by $X'$.
	Note that $X_0$ is still invariant under the action of $H_0$ and that $\card{X_0} = \card{X}$.
	Let $H_1 \subs \aut{\Wu}$ be such that $H_0 \subs H_1$, $\card{H_1} \leq \card{X}$, and every isomorphism of finitely generated substructures of $X_0$ extends to an automorphism from $H_1$.
	
	We continue this way by induction: given $X_n$ and $H_{n + 1}$ we let $X_{n + 1}$ be the substructure of $\Wu$ generated by the closure of $X_n$ under the action of $H_{n + 1}$, and we let $H_{n + 2} \supseteq H_{n + 1}$ be of cardinality $\leq \card{X}$ and extending every isomorphism of finitely generated substructures of $X_{n + 1}$.
	
	Finally, we put $U = \bigcup_{n \in \omega} X_n$ and $H = \bigcup_{n \in \omega} H_n$.
	Clearly, $X \subs U \leq \Wu$ and $\card{U} = \card{X}$ and $\card{H} \leq \card{X}$.
	Every isomorphism of finitely generated substructures of $U$ is already captured by some $X_n$, and so extends to $h \in H_{n + 1} \subs H$.
	Since $h[X_m] = X_m$ for every $m > n$, we have that $h \rest U$ is a well-defined automorphism of $U$.
	Hence, $U$ is invariant under the action of $H$, and the action witnesses that $U$ is homogeneous.
\end{pf}

\begin{prop} \label{thm:chain_necessary}
	Let $\Wu$ be an uncountable homogeneous structure of cardinality $\lambda$ with age $\Ef$.
	For every infinite $X \subs \Wu$ of cardinality $\lambda_0 < \lambda$ there is an extensible chain $\set{U_\alpha, G_\alpha, s_\alpha^\beta}_{\alpha<\beta\leq\lambda}$ such that $U_0 \supseteq X$, $U_\lambda = \Wu$, and $\min\set{\alpha \leq \lambda: \card{U_\alpha} = \kappa} = \kappa$ for every cardinal $\kappa \in (\lambda_0, \lambda]$.
\end{prop}
\begin{pf}
	First we define strictly increasing continuous chains $\set{U_\alpha}_{\alpha\leq\lambda}$ and $\set{H_\alpha}_{\alpha\leq\lambda}$.
	By Lemma~\ref{thm:closing_off} we obtain $U_0$ and $H_0$.
	Let $\set{w_\alpha}_{\alpha \in \lambda}$ be an enumeration of $\Wu$.
	Given $U_\alpha$ and $H_\alpha$ we apply Lemma~\ref{thm:closing_off} to $U_\alpha \cup \set{w_\alpha, x_\alpha}$ for some $x_\alpha \in \Wu \setminus U_\alpha$ and to $H_\alpha$. This way we obtain $U_{\alpha + 1}$ and $H_{\alpha + 1}$.
	For $\alpha \leq \lambda$ limit we put $U_\alpha = \bigcup_{\beta < \alpha} U_\beta$ and $H_\alpha = \bigcup_{\beta < \alpha} H_\beta$.
	
	Note that $\card{U_{\alpha + 1}} = \card{U_\alpha}$ for every $\alpha < \lambda$, so the cardinality can increase only at limit steps.
	For limit $\alpha \leq \lambda$ we inductively have
	\[\textstyle
		\card{H_\alpha} \leq \sum_{\beta<\alpha} \card{H_\beta} \leq \sum_{\beta<\alpha} \card{U_\beta} = \card{\alpha} + \sup_{\beta<\alpha} \card{U_\beta} = \card{U_\alpha},
	\]
	where the equalities follow from standard cardinal arithmetics and are proved as in the proof of Lemma~\ref{thm:limit_extension}.
	Hence, the assumption $\card{H_\alpha} \leq \card{U_\alpha}$ gets preserved at limit steps, 
	we have $\min\set{\alpha \leq \lambda: \card{U_\alpha} = \kappa} = \kappa$ for every cardinal $\kappa \in (\lambda_0, \lambda]$, and $U_\alpha \subsetneq \Wu$ for every $\alpha < \lambda$, and so the construction may proceed.
	Moreover, $U_\lambda = \Wu$ as it contains every point $w_\alpha$.
	
	Every $U_\alpha$ is homogeneous and invariant under the action of $H_\alpha$, and the homogeneity is witnessed by $G_\alpha = \set{h\rest U_\alpha: h \in H_\alpha}$.
	This is true at $0$ and at successor steps by Lemma~\ref{thm:closing_off} and remains automatically true at limit steps by continuity.
	
	Finally we need to define the coherent family $\set{s_\alpha^\beta\maps G_\alpha \to G_\beta}_{\alpha < \beta \leq \lambda}$.
	Let $r_\alpha\maps H_\alpha \to G_\alpha$ denote the restriction $h \mapsto h \rest U_\alpha$ for every $\alpha \leq \lambda$.
	We shall define sections $e_\alpha\maps G_\alpha \to H_\alpha$, i.e. maps satisfying $r_\alpha \cmp e_\alpha = \id{G_\alpha}$, or equivalently $g \subs e_\alpha(g)$ for every $g \in G_\alpha$.
	Then we put $s^\beta_\alpha = r_\beta \cmp e_\alpha$ for every $\alpha < \beta \leq \lambda$.
	The coherence condition $s^\gamma_\beta \cmp s^\beta_\alpha = s^\gamma_\alpha$ follows from the coherence condition $e_\beta \cmp s^\beta_\alpha  = e_\alpha$, which we assure by an inductive definition of the sections $e_\alpha$.
	
	We simply put $e_\beta(h) = e_\alpha(g)$ whenever $h = s^\beta_\alpha(g)$ for some $\alpha < \beta$ and $g \in G_\alpha$ (and define $e_\beta(h)$ arbitrarily so that it is a section, for $h$ not of the form $s^\beta_\alpha(g)$), but we need to show that the definition is correct.
	First, $h = s^\beta_\alpha(g)$ uniquely determines $g$ as $g = r_\alpha(h)$ since $r_\alpha \cmp r_\beta = r_\alpha$.
	Second, if $s^\beta_\alpha(g) = s^\beta_{\alpha'}(g')$ for some $\alpha < \alpha' < \beta$ and $g \in G_\alpha$, $g' \in G_{\alpha'}$, then $s^{\alpha'}_\alpha(g) = g'$ since $s^\beta_\alpha = s^\beta_{\alpha'} \cmp s^{\alpha'}_\alpha$ by the induction hypothesis and since we may cancel the injective map $s^\beta_{\alpha'}$.
	Hence, $e_{\alpha'}(g') = e_{\alpha'}(s^{\alpha'}_\alpha(g)) = e_\alpha(g)$ by the induction hypothesis.
\end{pf}

\begin{wn}
	The family of cardinalities of uncountable homogeneous structures with age $\Ef$ is downwards closed.
\end{wn}

By combining Proposition~\ref{thm:chain_necessary} with Proposition~\ref{thm:countably_extensible} and with the fact that every countable homogeneous structure of age $\Ef$ is isomorphic to $\U$, we obtain the following characterization.

\begin{wn}
	There exists an uncountable homogeneous structure with age $\Ef$ if and only if there exists an extensible $\U$-chain of length $\omega_1$.
\end{wn}

\begin{question}
	Suppose $\emb\U$ is non-degenerate.
	Is there an extensible $\U$-chain of length $2$?
	Namely, does there exist a non-trivial embedding $e\maps \U \to \U$ that is $G$-extensible for some $G \subs \aut{\U}$ witnessing the homogeneity?
\end{question}

\subsection{Amalgamating infinite structures} \label{sec:sigma}

It is natural to ask when the amalgamation property holds for the class $\sig{\Ef}$ of countable infinite structures induced by the Fraïssé class $\Ef$.
Many classes have this property -- see Propositions~\ref{PropFinLangE} and~\ref{PROPforPointSeventin} below. 
In fact, the question whether this always happens was informally asked by Michael Pinsker in 2020, during the third author's visit at the TU Wien.
It was also asked much earlier (2011) by Ita\"\i\ Ben Yaacov on MathOverflow, see~\url{https://mathoverflow.net/questions/71389}.
Some counterexamples have already appeared there, including the one from Section~\ref{SecRejnbouss} below. The next fact most likely belongs to the folklore of model theory; we present the proof just for the sake of completeness.

\begin{prop}\label{PropFinLangE}
	Suppose that $\Ef$ is a \fra\ class in a finite relational language.
	Then the class $\sig{\Ef}$ of all countable structures with age in $\Ef$ has the amalgamation property.
\end{prop}

\begin{pf}
	Fix $Z \in \sig{\Ef}$ and fix embeddings $\map f Z X$, $\map g Z Y$ such that $X,Y \in \sig{\Ef}$ and $X \setminus \img fZ$, $Y \setminus \img gZ$ are singletons.
	We may even assume that the embeddings are inclusions, and $X = Z \cup \sn a$, $Y = Z \cup \sn b$.
	Write $Z = \bigcup_{\ntr}Z_n$, where $Z_0 \subs Z_1 \subs \cdots$ and $Z_n \in \Ef$ for every $\ntr$.
	We know that for every $\ntr$ the embeddings $Z_n \subs Z_n \cup \sn a$, $Z_n \subs Z_n \cup \sn b$ can be amalgamated, namely, there is a structure on $Z_n \cup \dn ab$ extending both of them (possibly identifying $a$ and $b$).
	Since the language is finite, there are only finitely many possibilities.
	Given an amalgamation of $Z_n \cup \sn a$, $Z_n \cup \sn b$, for every $k<n$ there is a uniquely determined amalgamation of $Z_k \cup \sn a$, $Z_k \cup \sn b$, as these are just substructures of $Z_n \cup \dn ab$.
	Thus we can order all the amalgamations this way, obtaining a tree with finite levels. The tree is obviously infinite, therefore by K\H{o}nig's Lemma it has an infinite branch that corresponds to an amalgamation of $X$ and $Y$.
	
	Finally, we proceed by induction, as every embedding of $Z$ into a countable structure can be decomposed into a sequence of one-point extensions.	
\end{pf}

\begin{wn}
	Let $\Ef$ be a \fra\ class in a finite relational language. Then there exists a homogeneous structure of cardinality $\aleph_1$ whose age is $\Ef$.
\end{wn}

The situation is totally different when we allow algebraic operations or infinitely many relations. We present examples in Section~\ref{SECeroigjwroigj}; one of them is actually classical and well-known: the class of countable locally finite groups.

\bigskip

It is well-known that if the Fraïssé class $\Ef$ has a finite relational language, then $\U$ is $\omega$-categorical, and that if $\U$ is $\omega$-categorical, then it is $\omega$-saturated, which is enough for the amalgamation property of countable structures.

\begin{prop}\label{PROPforPointSeventin}
	If $\U$ is $\omega$-saturated, then $\sig\Ef$ has the amalgamation property.
\end{prop}
\begin{proof}
	Let $X \geq A \leq Y$ be structures from $\sig\Ef$.
	Without loss of generality, $X \leq \U$ and $Y$ is generated by $A \cup \set{y}$ for some $y \in Y$.
	Let $\Wu$ be an $\omega_1$-saturated elementary extension of $\U$,
	and let $p(y)$ be the quantifier-free type of $y$ in $Y$ over $A$.
	We have that $p(y)$ is a type in $\Wu$ over $A$ since for every finite $B \subs A$ the quantifier-free formulas from $p(y)$ using only parameters from $B$ are realized by $y$ in the substructure of $Y$ generated by $B \cup \set{y}$, and so also by its isomorphic copy in $\U$ pointwise fixing $B$ (which exists as $\U$ is injective).
	Since $\Wu$ is $\omega_1$-saturated, $p(y)$ is realized by some $y' \in \Wu$.
	Hence, $Y' \leq \Wu$ generated by $A \cup \set{y}$ is an isomorphic copy of $Y$, and so the substructure $Z$ generated by $X \cup Y'$ is an amalgamation of $X \geq A \leq Y$.
	
	It remains to show that $Z \in \sig\Ef$, which reduces to $\Fin(Z) \subs \Ef$.
	It is enough to show that $\Fin(\Wu) = \Fin(\U)$.
	Since $\Wu$ is elementarily equivalent to $\U$ and $\U$ is $\omega$-saturated, for every $W \leq \Wu$ generated by a finite tuple $\bar{w}$ there is a tuple $\bar{w}'$ in $\U$ such that $\seq{\U, \bar{w}'}$ is elementarily equivalent to $\seq{\Wu, \bar{w}}$, and so the substructure of $\U$ generated by $\bar{w}'$ is isomorphic to $W$, hence $W \in \Fin(\U)$.
\end{proof}

\section{Examples}\label{SECeroigjwroigj}

Below we give relevant examples related to the preservation of AP, and in particular an example of a Fraïssé limit in a countable relational language that is not an amalgamation base, but Proposition~\ref{thm:countably_extensible} applies, and so there is an uncountable homogeneous structure with the same age (Section~\ref{sec:main_example}).

The first two examples belong to the folklore of group theory.
The first one actually presents a trivial monoid of embeddings.

\subsection{Locally cyclic groups}\label{SubSectgoijwegijw}

Let $\Gee$ be the class of all cyclic groups.
Then $\Cl{\Gee}$ is the class of all
locally cyclic groups, namely, groups in which every finite (equivalently: at most two-element) set generates a cyclic group.
It is well-known that, up to isomorphism, $\Cl\Gee$ consists of all quotients of subgroups of the rationals $\Qyu$. In particular, $\Cl\Gee$ does not contain any uncountable group. Obviously, the class $\Gee$ fails the amalgamation. For example, no locally cyclic group can contain both $\Zee$ and $\Zee_n$ with $n \in \Nat\setminus \dn01$.

Now let $\Gee_0$ consist of all \emph{torsion-free} cyclic groups. So, up to isomorphism, 
$$\Gee_0 = \dn {\sn0}{\Zee}.$$
Then $\Gee_0$ is a \fra\ class and its \fra\ limit is $\Qyu$. Nevertheless, the monoid $\emb{\Qyu}$ is degenerate, as every self-embedding of $\Qyu$ a multiplication by some rational $q$ and so is an automorphism.
On the other hand, there is a canonical Kat\v etov functor $K$ on $\Gee_0$, namely, $K(\{0\}) = K(\Zee) = \Qyu$, and the natural transformation is just the inclusion. Every embedding in $\Gee_0$ corresponds to a multiplication of the nonnegative generator, which extends uniquely to an embedding---again multiplying by the same element, now in $\Qyu$.

\subsection{Locally finite groups}

Consider the class $\Gee$ of all finite groups.
Recall that a group is \define{locally finite }{locally finite group} if all its finitely generated subgroups are finite.

It is a non-trivial fact that $\Gee$ has the amalgamation property (see \cite{Hall} or \cite{Neu}).
On the other hand, $\sig \Gee$ fails the amalgamation property, as the following example (extracted from \cite{GroShe}, although most likely much older) shows.

Let $G$ be the group of all permutations of $\Nat$ that have a finite support. Namely, $f \in G$ if the set
$$\setof{n \in \Nat}{f(n) \ne n}$$
is finite.
Clearly, $G$ is a locally finite countable group, therefore it belongs to $\sig \Gee$.
Let $\map a \Nat \Nat$ be the permutation defined by $a(2k) = 2k+1$, $a(2k+1) = 2k$ for every $k \in \Nat$.
Next, let $\map b \Nat \Nat$ be defined by $b(0)=0$, $b(2k-1)=2k$ and $b(2k) = 2k-1$ for every $k>0$.
Let $G[a]$, $G[b]$ denote the groups generated by $G \cup \sn a$ and $G \cup \sn b$, respectively (as subgroups of the full permutation group of $\Nat$).
It is easy to verify that both $G[a]$ and $G[b]$ are locally finite, while in any amalgamation $a b$ has an infinite order:
we have $b a (2k\ 2k + 2) a^{-1} b^{-1} = (b(a(2k))\ b(a(2k + 2))) = (2k + 2\ 2k + 4)$ in every amalgamation.

This example leads to the following question.

\begin{question}
	Does there exist an uncountable homogeneous locally finite group?
\end{question}

Note that there exist locally finite groups of arbitrary cardinality.
The question whether there exists a Kat\v{e}tov functor on the class of finite groups seems to be open, although results in the positive direction can be found~\cite{Sh312}.

\subsection{Anti-metric spaces}

Here we describe a simple example of a relational \fra\ class $\Aaa$ for which $\sig \Aaa$ totally fails the amalgamation property.

Namely, a \emph{natural anti-metric space} is a pair $(X,\rho)$ such that $\map \rho {X^2} \Nat$ is a symmetric function, $\rho(x,y)=0$ if and only if $x=y$ and the triangle inequality fails everywhere. The last condition means for every pairwise distinct $x,y,z \in X$, if
$$\rho(x,y) \goe \max \{ \rho(x,z), \rho(z,y)\}$$
then $\rho(x,y) > \rho(x,z) + \rho(z,y)$.
In other words, if $0 < a \loe b \loe c$ denote ``distances'' in a triangle then $c > a+b$. 

Let $\Aaa$ be the class of all finite anti-metric spaces and, as usual, let $\sig \Aaa$ be the class of all countable ones.
It is clear that $\Aaa$ is a \fra\ class that can be encoded in a countable relational language.
The amalgamation property of one-point extensions can be easily seen, noticing that one can also define the ``distance'' between the pair of ``new'' points to be a large enough natural number.

\begin{prop}
	The amalgamation property fails everywhere in $\sig{\Aaa} \setminus \Aaa$, i.e. no $X \in \sig{\Aaa} \setminus \Aaa$ is an amalgamation base.
\end{prop}

\begin{pf}
	Fix a countable infinite anti-metric space $(X,\rho)$. We may assume that $X = \nat$.
	Let $X_a = X \cup \sn a$, where the anti-distance from $a$ is defined by induction as follows.
	We set $\rho(a,0) = 1$.
	Having defined $\rho(a,k)$ for $k < n$, we define $\rho(a,n)$ to be a fixed natural number strictly greater than
	$$\rho(a,i)+\rho(i,n)$$ for all $i<n$.
	By this way, $\{0, \dots, n, a\}$ is an anti-metric space for each $\ntr$ and finally $X_a$ becomes an anti-metric space.
	Next, let $X_b = X \cup \sn b$, where $b \notin X_a$ and define $\rho(b,n) = \rho(a,n)+1$.
	Again, $X_b$ is an anti-metric space.
	
	Suppose $Y = X \cup \dn a b$ is an amalgam of $X_a$, $X_b$ over $X$. Note that we are not allowed to glue $a$ and $b$, because the distances are different.
	On the other hand, if $\rho(a,b) = k$ then for a big enough $n$ we have $\rho(a,n) > k$ and $\rho(b,n) = \rho(a,n)+1 > k$, showing that that triangle inequality between $a,b,n$ actually holds, a contradiction.
\end{pf}

\subsection{Complete labeled graphs with no triangles}\label{SecRejnbouss}

This is an example due to Greb\'\i k~\cite{Gre}, answering a question on the existence of Kat\v{e}tov functors (see~\cite{KubMas}).

Namely, a \emph{triangle-free complete graph} is a complete graph $G$ in which each edge has assigned a fixed natural number (color) and there is no monochromatic triangle.
Denote by $\Gee$ the class of all finite triangle-free complete graphs.
Clearly, this is a \fra\ class in a countable relational language. The amalgamation property follows from the fact that finite graphs admit finitely many colors, therefore one can always amalgamate one-point extensions by assigning a new color to the new pair of vertices.

\begin{prop}
	The amalgamation property fails everywhere in $\sig{\Gee} \setminus \Gee$.
\end{prop}

\begin{pf}
	Let $X \in \sig\Gee \setminus \Gee$.
	There is a point $\infty \in X$ such that for two points $x \neq y \in X \setminus \sn \infty$ the edges $\dn x \infty$ and $\dn y \infty$ have different colors.
	Without loss of generality, $X = \omega \cup \sn\infty$, $x = 0$, $y = 1$, and the color of $\dn 0 \infty$ is $1$, while the color of $\dn 1 \infty$ is $0$.
	We consider two one-point extensions $X \cup \sn a$ and $X \cup \sn b$, defined as follows.
	Every edge $\dn i k$, $i \in \dn a b$, $k \in \omega$, has color $k$, and the edge $\dn a \infty$ has color $0$, while $\dn b \infty$ has color $1$.
	
	The extensions are triangle-free as the only possibly monochromatic triangles are $\set{a, 0, \infty}$ and $\set{b, 1, \infty}$.
	The extensions cannot be identified as the colors of $\dn a \infty$ and $\dn b \infty$ disagree.
	Finally, coloring the edge $\dn a b$ by $n$ will give rise to a monochromatic triangle $\set{a, b, n}$.
\end{pf}

Note that every natural anti-metric space can be regarded as a complete labeled graph by treating the distance $\rho(x, y) \in \Nat$ as the color of the edge $\{x, y\}$ for $x \neq y$.
Obviously, the resulting graph will have no monochromatic triangles.

The last two examples (anti-metric spaces and complete labeled graphs) have a bound on cardinalities of structures. Namely, due to the Erd\H{o}s--Rado Theorem, a structure of cardinality $> 2^{\aleph_0}$, whose all pairs have labels from a countable set, admits a triangle (even an infinite substructure) that is monochromatic, namely, all the labels are the same. This clearly shows that there is no anti-metric space of cardinality strictly greater than the continuum. Similarly, there is no triangle-free complete graph of cardinality greater than the continuum.
There exists a natural anti-metric space of cardinality $2^{\aleph_0}$. Indeed, consider $2^\omega$ and define $\rho(x,y) = 3^n$, where $n$ is the minimal coordinate such that $x(n) \ne y(n)$.

\subsection{The linearly ordered variant} \label{sec:main_example}

We consider a linearly ordered modification of the previous example. In fact, the ordered version is the actual example of Grebík~\cite{Gre}.
Formally, by an ordered complete labeled graph we mean a structure $\seq{X, c_X, <_X}$ where $X$ is a set, $\map {c_X} {\dpower{X}{2}} \omega$ is a coloring of all edges, and $<_X$ is a linear order on $X$.
As usual we often omit the subscripts of $c_X$ and $<_X$, and we identify $X$ with the whole structure.

Let $\Gee^<$ denote the class of all finite triangle-free ordered complete labeled graphs, so $\sig\Gee^<$ consists of all countable triangle-free ordered complete labeled graphs, and $\Cl{\Gee^<}$ consists of all triangle-free ordered complete labeled graphs.
It is easy to see that $\Gee^<$ is a Fraïssé class.
Let $U$ denote its limit.
We have the following.
\begin{enumerate}
	\item No $X \in \sig\Gee^< \setminus \Gee^<$ is an amalgamation base in $\sig\Gee^<$, in particular $\emb{U}$ does not have the amalgamation property.
	This is easy to see: consider two one-point extensions $X \cup \set{a}$ and $X \cup \set{b}$ such that $c(\set{x, a}) = c(\set{x, b})$ attains all values in $\omega$ for $x \in X$, while for some $x \in X$ we have $a < x$ in $X \cup \set{a}$, but $x < b$ in $X \cup \set{b}$.
	Hence, the points $a$ and $b$ cannot be glued and putting $c(\set{a, b}) = q$ would introduce a monochromatic triangle for every $q \in \omega$.
	\item There is no Katětov functor $\sig\Gee^< \to \sig\Gee^<$. This was proved by Grebík~\cite{Gre}.
	\item \label{itm:colors_extensible} For every countable subgroup $G \leq \aut U$ there is a non-trivial $G$-extensible embedding $\map e U U$, and so there is an uncountable homogeneous triangle-free ordered complete labeled graph.
\end{enumerate}
We adapt the proof of universality of $\aut U$ from \cite{Gre} to show \ref{itm:colors_extensible}.

A one-point extension of $X \in \Cl{\Gee^<}$ is without loss of generality of the form $X \cup \set{\alpha}$ where $\map \alpha X \omega$ is the map defined by $\alpha(x) = c(x, \alpha)$, so to fully describe $X \cup \set{\alpha}$ it is enough to remember the map $\alpha$ and the linear order on $X \cup \set{\alpha}$ denoted by $<_\alpha$.
A one-point extension $\alpha$ of $X$ is called 
\begin{itemize}
	\item \emph{$Q$-colored} for $Q \subs \omega$ if $\map \alpha X \omega$ takes values only in $Q$,
	\item \emph{finitary} if every color support $\alpha^{-1}(q)$ for $q \in \omega$ is finite.
\end{itemize}

\begin{lm}
	For every $X \in \Cl{\Gee^<}$ and co-infinite $Q \subs \omega$ there is an extensible embedding $\map e X E(X)$ for some $E(X) \in \Cl{\Gee^<}$ with an extension operator $\map E {\aut{X}} {\aut{E(X)}}$ such that every finitary $Q$-colored one-point extension of $X$ is uniquely realized in $E(X)$.
\end{lm}
\begin{proof}
	Let $E(X) = X \cup \set{\alpha: \alpha$ a finitary $Q$-colored one-point extension of $X}$ and let $\map e X {E(X)}$ be the inclusion.
	Clearly the canonical action of $\aut X$ on $X$ can be naturally extended to an action on $E(X)$: for $g \in \aut X$ and $\alpha \in E(X) \setminus X$ we let $g(\alpha)$ be the coloring with fibers $g(\alpha)^{-1}(q) = g[\alpha^{-1}(q)]$ for every color $q \in \omega$ so that $c(\set{x, \alpha}) = c(\set{g(x), g(\alpha)})$, and we let $<_{g(\alpha)}$ be so that $x <_\alpha \alpha \Leftrightarrow g(x) <_{g(\alpha)} g(\alpha)$ for every $x \in X$.
	This defines the extension $E(g)$ and the extension operator $E$.
	
	It remains to turn $E(X)$ into a triangle-free ordered complete labeled graph so that every $E(g)$ is an automorphism of $E(X)$.
	First we define a linear order on $E(X)$.
	The order is already defined on every $X \cup \set{\alpha}$.
	We extend the linear order on $X$ to a linear order on $\fin{X}$: for $A \neq B \in \fin{X}$ we put $A <' B$ if $|A| < |B|$ or if $|A| = |B| = n$ and $a_j < b_j$ where $\set{a_i: i < n} = A$ and $\set{b_i: i < n} = B$ are $<$-increasing enumerations of $A$ and $B$ and $j = \min\set{i < n: a_i \neq b_i}$.
	Then for $\alpha \neq \beta \in E(X) \setminus X$ we put $\alpha < \beta$ if there is $x \in X$ with $\alpha <_\alpha x <_\beta \beta$ or if $\set{x \in X: x <_\alpha \alpha} = \set{x \in X: x <_\beta \beta}$ and $\alpha^{-1}(r) <' \beta^{-1}(r)$ where $r = \min\set{q \in \omega: \alpha^{-1}(q) \neq \beta^{-1}(q)}$. Here we are using that the color supports $\alpha^{-1}(q)$ and $\beta^{-1}(q)$ are finite.
	This is a well-defined linear order that is invariant under automorphisms of $X$ (i.e. $\alpha < \beta \Leftrightarrow g(\alpha) < g(\beta)$ for every $g \in \aut X$) since the base orders $<_\alpha$, $<_\beta$, and $<'$ are invariant.
	
	Finally we need to color all the edges $\set{\alpha, \beta}$ for $\alpha < \beta \in E(X) \setminus X$ in a way such that every $g \in \aut X$ still induces an automorphism of $E(X)$ while not creating any monochromatic triangles.
	The invariance under automorphisms means that $c(\set{g(\alpha), g(\beta)}) = c(\set{\alpha, \beta})$, i.e. every orbit of the action of $\aut(X)$ on two-point subsets of $E(X) \setminus X$ has to be monochromatic.
	Since $\omega \setminus Q$ is infinite, we can color edges in different orbits by different colors from $\omega \setminus Q$.
	This also means that monochromatic triangles may exist only in $E(X) \setminus X$.
	
	We show that there are no monochromatic triangles in $E(X) \setminus X$.
	Suppose for contradiction that $\alpha < \beta < \gamma$ is a monochromatic triangle.
	So there is an automorphism $g \in \aut E(X)$ that maps $\set{\alpha, \beta} \to \set{\alpha, \gamma}$.
	Since $g$ preserves the order, we have $g(\alpha) = \alpha$.
	So $g$ maps the finite linearly ordered set $\alpha^{-1}(q)$ for $q \in \omega$ onto itself, and hence $g$ fixes $\alpha^{-1}(q)$ pointwise.
	Since $q \in \omega$ was arbitrary, we have $g = \id{X}$ and $\gamma = g(\beta) = \beta$, which is a contradiction.
\end{proof}

Of course even for a countable graph $X$ the extension $E(X)$ has cardinality continuum.
However we can stay countable if we restrict ourselves to countably many automorphisms and countably many one-point extensions, which turns out to be enough for our purposes.
\begin{wn} \label{omega_graphs_countable_extension}
	For every $X \in \sig\Gee^<$, countable $G \leq \aut X$, co-infinite $Q \subs \omega$, and a countable family of finitary $Q$-colored one-point extensions $A$, there is a $G$-extensible embedding $\map f X Y$ for some $Y \in \sig\Gee^<$ with an extension operator $G \to \aut{Y}$ such that every one-point extension $\alpha \in A$ is uniquely realized in $Y$.
\end{wn}
\begin{proof}
	Let $\map e X {E(X)}$ be the extensible embedding from the previous lemma.
	Without loss of generality $A \subs E(X) \setminus e[X]$.
	It is enough to take $Y = X \cup \set{g(\alpha): \alpha \in A, g \in G}$, i.e. we close $A$ under the action of our countable subgroup $G$.
	Then we let $\map f X Y$ be the restriction of $\map e X {E(X)}$.
\end{proof}

\begin{proof}[Proof of \ref{itm:colors_extensible}]
	Let $G \leq \aut U$ be a countable subgroup.
	Let us write $\omega$ as an increasing union of subsets $\bigcup_{n \in \omega} Q_n$ such that every $Q_{n + 1} \setminus Q_n$ is infinite.
	We put $X_0 = U$ and $G_0 = G$.
	Suppose $X_n \in \sig\Gee^<$ and countable $G_n \leq \aut X_n$ are defined.
	Let $A_n$ be a countable family of finitary $Q_n$-colored one-point extensions of $X_n$ such that every $Q_n$-colored one-point extension $\alpha_0$ of every finite substructure $F \subs X_n$ can be extended to an extension $\alpha \in A_n$.
	This is possible: we put $\alpha(x) = \alpha_0(x)$ for $x \in F$ and for every $x \in X \setminus F$ we use a different color from $Q_n \setminus \alpha_0[F]$ for $\alpha(x)$.
	Hence $X \cup \set{\alpha}$ is triangle-free.
	We can extend the order e.g. by putting $x <_\alpha \alpha$ if there is $y \in F$ such that $x \leq y <_{\alpha_0} \alpha_0$, so $\alpha$ becomes the immediate successor of $\max\set{y \in F: y <_{\alpha_0} \alpha_0}$ in $X \cup \set{\alpha}$.
	
	Let $\map {e_n} {X_n} {X_{n + 1}}$ be a $G_n$-extensible embedding with an extension operator $\map {E_n} {G_n} {\aut X_{n + 1}}$ obtained from Corollary~\ref{omega_graphs_countable_extension} applied to $X_n$, $G_n$, $Q_n$, and $A_n$, and let $G_{n + 1} = E[G_n] \leq \aut X_{n + 1}$.
	Without loss of generality, $X_{n + 1}$ uses only colors from $Q_{n + 1}$ since otherwise we may change the coloring $c_{X_{n + 1}}$ so that every use of a color from $\rng(c_{X_{n + 1}}) \setminus Q_n$ is replaced by a corresponding color from the infinite set $Q_{n + 1} \setminus Q_n$.
	This concludes the inductive definition.
	Let $X_\infty$ and the family of embeddings $\map {e_n^\infty} {X_n} {X_\infty}$ be the colimit of the sequence $\vec{e}$.
	The embedding $\map {e_0^\infty} U {X_\infty}$ is non-trivial (as every $e_n$ is non-trivial) and $G$-extensible: for every $g \in G$ the maps $E_0^n(g) \in \aut X_n$ are extending each other via the bonding maps $e_n$, and so the colimiting map $g_\infty \in \aut X_\infty$ satisfies $g_\infty \cmp e_0^\infty = e_0^\infty \cmp g$.
	
	It is enough to observe that $X_\infty$ is isomorphic to $U$.
	We show that $X_\infty$ is injective.
	Let $\alpha_0$ be a one-point extension of a finite substructure $F \subs X_\infty$.
	There is $n \in \omega$ such that $F \subs X_n$ and $\alpha_0$ is $Q_n$-colored.
	Hence, there is $\alpha \in A_n$ extending $\alpha_0$, and so $\alpha_0$ is realized in $X_{n + 1}$.
	
	Therefore, there is an uncountable homogeneous triangle-free ordered complete labeled graph by Proposition~\ref{thm:countably_extensible}.
\end{proof}

\section{Final remarks}

Let us note that our main result can be proved in a much more general setting, replacing embeddings by more abstract morphisms. The only ingredient needed is that partial isomorphisms between ``small'' structures are ``captured'' in intermediate steps. Of course, one also needs to specify what ``uncountable'' means. A good framework is an $\omega$-accessible category, in the sense of Grothendieck~\cite{Groth} (complemented by Gabriel and Ulmer~\cite{GabUlm}). In this setting finitely generated structures are finitely presentable objects and morphisms from them factor through directed colimits. The monoid of self-embeddings is simply a hom-set of the form $\operatorname{Hom}(X,X)$ and all morphisms are typically assumed to be monomorphisms.

Coming back to our concrete setting, let us recall the main intriguing questions left open:

\begin{question}
	Assume $\Ef$ is a \fra\ class such that there exist uncountable structures with age $\Ef$. Does there exist an uncountable homogeneous structure with age $\Ef$?
\end{question}

\begin{question}
	Assume $\U$ is a \fra\ limit such that $\emb{\U}$ is non-degenerate. Does there exist a non-trivial extensible self-embedding of $\U$?
\end{question}

In fact, we do not know the answer to a weaker question, asking for $G$-extensibility for every countable $G \loe \aut{\U}$.

\separator

\paragraph{Acknowledgments.} 

We would like to acknowledge that Mirna Džamonja was a co-author of the originally submitted version of the paper. The authors are extremely grateful to Mirna Džamonja for bringing the main question and for numerous discussions.

Research of Adam Bartoš and Wiesław Kubiś was supported by GA ČR (Czech Science Foundation) grant EXPRO 20-31529X and by the Czech Academy of Sciences (RVO 67985840).

The authors would like to thank Nate Ackerman, David Bradley-Williams, and Christian Pech for fruitful discussions during the Midsummer Combinatorial Workshop XXIX (July 29 -- August 2, 2024, Prague), in particular, leading to Proposition~\ref{PROPforPointSeventin}. Special thanks are due to the Organizers of the workshop mentioned above.


\end{document}